\documentclass[12pt]{amsart}
\usepackage{amsthm,amsmath,amssymb}
\usepackage{xypic}
\usepackage{enumerate}

\newcommand{\al}{\alpha}
\newcommand{\be}{\beta}
\newcommand{\de}{\delta}
\newcommand{\e}{\varepsilon}
\newcommand{\CC}{{\mathbb{C}}}

\newcommand{\HH}{{\mathbb{H}}}
\newcommand{\FF}{{\mathbb{F}}}

\newcommand{\PP}{{\mathbb{P}}}

\newcommand{\NN}{{\mathbb{N}}}
\newcommand{\ZZ}{{\mathbb{Z}}}

\newcommand{\calA}{{\mathcal A}}

\newcommand{\calX}{{\mathcal X}}

\def\ZZ{{\mathbb Z}}

\def\PP{{\textbf P}}

\newcommand{\op}{\operatorname}

\newcommand{\SpZ}{\op{Sp}(2g,\ZZ)}

\newcommand{\T}{\Theta}
\renewcommand{\t}{\theta}

\renewcommand\tt[2]{\t\left[\begin{matrix}#1\\ #2\end{matrix}\right]}

\def\tch#1#2{{\left[\begin{matrix}#1\\ #2\end{matrix}\right]}}
\def\tt#1#2{{\t\tch{#1}{#2}}}
\newcommand{\s}{\sigma}

\def\aa{\overline{\mathcal{A}}}

\theoremstyle{plain}
\newtheorem{thm}{Theorem}
\newtheorem{lm}[thm] {Lemma}
\newtheorem{prop}[thm]{Proposition}
\newtheorem{cor}[thm]{Corollary}

\theoremstyle{definition}

\newtheorem{qu}[thm]{Question}

\begin{document}
\title{On the Coble quartic and Fourier-Jacobi expansion of  theta relations }

\author{Francesco Dalla Piazza}
\address{Universit\`a ``La Sapienza'', Dipartimento di Matematica, Piazzale A. Moro 2, I-00185, Roma,   Italy}
\email{dallapiazza@mat.uniroma1.it,\, f.dallapiazza@gmail.com}
 
\author{Riccardo Salvati Manni}
\address{Universit\`a ``La Sapienza'', Dipartimento di Matematica, Piazzale A. Moro 2, I-00185, Roma,   Italy}
\email{salvati@mat.uniroma1.it}
\thanks{Research of the authors is supported in part by  Cofin 2011}

\begin{abstract} In \cite{sturmfels}, the  authors conjectured equations  for the universal Kummer variety in genus  3 case.  Though, most of these equations are obtained from  the Fourier-Jacobi expansion of relations among theta constants in genus 4,  the more prominent one,  Coble's  quartic, cf. \cite{coblebook}  was obtained differently, cf. \cite{gscoble} too.

The aim of the current paper is to  show that Coble's quartic  can be obtained as  Fourier-Jacobi expansion of a relation among theta-constants in genus 4. We get also one more relation that could be  in the ideal described in \cite{sturmfels}.
\end{abstract}
\maketitle

\section{Introduction}
 Let $\calX_g(2,4)$  be the universal Kummer  variety  with a suitable level  structure.

We are interested in the map

$$\kappa_g:\calX_g(2,4)\to \PP^{2^g-1}\times  \PP^{2^g-1},$$
given by
$$(\tau,z)\mapsto (\dots, \T[\e](\tau),\dots)\times(\dots, \T[\e](\tau, z),\dots).$$

By basic facts about theta functions we  know that this map is well defined and generically injective. We would like to  discuss equations  for the image.

First of all we have to recall that the projections on the two factors of the product define the maps

$$\T_g:\calA_g(2,4)\to \PP^{2^g-1},\quad {\rm  given \, by} \, \tau\mapsto  (\dots, \T[\e](\tau),\dots)$$
and

$$\psi_g:K_{\tau}:=X_{\tau}/\pm 1\to \PP^{2^g-1},\quad {\rm  given \, by} \, \tau\mapsto  (\dots, \T[\e](\tau, z),\dots).$$
 The first map is generically injective,  the second is finite of order $2^{k-1}$ with $k$  equal to the  number
of irreducible blocks of the  class of the period matrix $\tau$. In the irreducible case we have that the second map is an immersion, cf.
\cite{kempf} and it is known that the ideal of the image  is generated by forms of degree at most 4, cf.  \cite{tirabassi}.  So we can  have equations for the Kummer variety. Quartic equations are of  special interest.   In genus 2 case there is only one  quartic equation. It has been determined in several different  way, cf, \cite{khaled}, \cite{sturmfels}
 and \cite { gscoble}. In particular, in \cite{sturmfels} one method is described, explained  by Grushevsky and the second  author, that uses Fourier-Jacobi expansion of  the  unique relation  among  second order theta constants of genus 3.
\smallskip
\\
In genus 3 case the situation is  more involved. In fact there is a peculiar equation: Coble's  quartic in second order theta functions.
This has the  property that if  $\tau$  is   the  period matrix of  the jacobian variety of a smooth plane  quartic, equivalently theta constants of the first order do not vanish  at  $ \tau$,  then  the singular  locus
of the Coble's quartic is scheme theoretic isomorphic to the Kummer  variety, cf. \cite{laszlo}.
In the  two recent  papers \cite{sturmfels}, \cite{gscoble}  has been given an explicit equation for  Coble's  quartic. The  method used was restriction of a suitable equation to the   boundary  components of the Satake compactification   of   $\aa_3(2,4)$  in \cite{sturmfels}  and   the  link  between  plane quartics and  projective invariants of   seven points of $\PP^2$ in \cite{gscoble}.\smallskip

More quartic equations are discussed in \cite{sturmfels}.  All of them are obtained  from  Fourier-Jacobi expansion  of relations in genus 4 obtained  turning into   relations between  $\T[\al](Z)$ the Riemann relations. The aim of the current paper is  to get Coble's quartic and possibly  other relevant  quartics as Fourier-Jacobi expansion of   suitable relations among  $\T[\al](Z)$  of genus 4.

\section*{Acknowledgements}
The authors thanks Bert van Geemen,  Samuel Grushevsky, Manabu Oura,  Cris Poor and David Yuen  for interesting conversations.  They are  grateful to Samuel Grushevsky for his help on a  first  version of this manuscript.

\section{Notations and definitions}
We denote by $\HH_g$ the Siegel upper half-space of symmetric complex matrices with positive-definite imaginary part.  An element $\tau\in\HH_g$ is called a {\it period
matrix}, and defines the complex abelian variety $X_\tau:=\CC^g/\ZZ^g+\tau
\ZZ^g$. The group $\Gamma_g:={\rm Sp}(2g,\ZZ)$ acts on $\HH_g$ by
automorphisms: for $M:=\begin{pmatrix} a&b\\
c&d\end{pmatrix}\in{\rm Sp}(2g,\ZZ)$ the action is
$$M\cdot\tau:=(a\tau+b)(c\tau+d)^{-1}.$$
 A period
matrix $\tau$ is called {\it reducible} if there exists
$M\in\Gamma_g$ such that
$$
M\cdot\tau=\begin{pmatrix} \tau_1&0\\
0&\tau_2\end{pmatrix},\quad\tau_i\in\HH_{g_i},\ g_1+g_2=g;
$$
otherwise we say that $\tau$ is irreducible.
The quotient of $\HH_g$ by
the action of the symplectic group is $\calA_g$, the moduli space of
principally polarized abelian varieties (ppavs).
This quotient is an analytic variety with finite quotient singularities, and by a well-known result of Satake  also a quasi-projective variety. Indeed, it is the coarse moduli space associated to the moduli stack of ppav.\smallskip

We consider the group $\Gamma_g\ltimes\ZZ^{2g}$  where the semi-direct product is given by the natural action of $\Gamma_g$ on vectors of length 2g. This group acts on $\HH_g \times \CC_g$ by
$$
\left(M,\begin{pmatrix} m\\
n\end{pmatrix}\right)\cdot(\tau,z)=(M\cdot\tau,(z+ \tau m+n)(c\tau+d)^{-1}).$$
We would like to say that the quotient $$\calX _g =\Gamma_g\ltimes\ZZ^{2g}/\HH_g\times \CC^g $$is the universal abelian variety over $\calA_g$. This is true in the sense of stacks, but not for coarse moduli spaces, in fact we have that in particular $  \left(-1_{2g},\begin{pmatrix} 0\\
0\end{pmatrix}\right)$ acts on each fibre by the involution $z\mapsto\,-z$. Hence we get the universal   Kummer variety.
It  appears canonically in the toroidal compactifications
of  $\calA_{g+1}$. In fact  for any   toroidal compactification the  first  boundary component is  isomorphic to $\calX _g$.
Hence  a method to study the universal Kummer variety of genus $g$ leads to consider (partial)  toroidal  compactifications of
$\calA_{g+1}$. 
We would like
 to use theta functions to parametrize such varieties.
 Hence it is rather  natural to introduce some level structures that allow us to avoid the stabilizers and to  use theta functions.

We define the  subgroups of the symplectic group to be
\begin{align*}
\Gamma_g[n]&:=\left\lbrace M=\begin{pmatrix} a&b\\ c&d\end{pmatrix}
\in\Gamma_g\, |\, M\equiv\begin{pmatrix} 1&0\\
0&1\end{pmatrix}\ {\rm mod}\ n\right\rbrace, \\
\Gamma_g[n,2n]&:=\left\lbrace M\in\Gamma_g(n)\, |\, {\rm diag}(a^tb)\equiv{\rm diag}
(c^td)\equiv0\ {\rm mod}\ 2n\right\rbrace.
\end{align*}
The corresponding {\it level moduli spaces of ppavs}  and  {\it level moduli spaces of ppavs with theta structure}  are denoted
$\calA_g(n)$ and $\calA_g(n,2n)$, respectively.

In the case of the $\calA_g(2,4)$ the stabilizer  group is $\pm 1_{2g}$ if and only if the corresponding period matrix is irreducible. We set $\calX_g (2,4)$ the  corresponding quotient for the action of the group
$$G_g(2,4)=\Gamma_g[2,4]\ltimes 2\ZZ^{2g}.$$

\section {theta functions}
For a period matrix $\tau\in\HH_g$ the principal polarization $\T_\tau$ on the abelian variety $A_\tau:=\CC^g/(\ZZ^g+\tau\ZZ^g)$ is induced  by  the divisor of the theta function
$$
  \theta(\tau,z):=\sum\limits_{n\in\ZZ^g}\exp(\pi i (^tn\tau n+2 ^tn
  z)).
$$
Notice that for fixed $\tau$ theta is a function of $z\in\CC^g$, and its automorphy properties under the lattice $\tau\ZZ^g+ \ZZ^g$ define the bundle $\T_\tau$.

Given a point of order two on $A_\tau$, which can be uniquely
represented as $\frac{\tau\e+\de}{2}$ for $\e,\de\in \FF_2^g$ (where $\FF_2=\lbrace 0,1\rbrace$ is the additive group),  
the {\it first  order
theta function with characteristic $m=[\e,\de]$} is
$$\tt\e\de(\tau,z):=
\sum\limits_{m\in\ZZ^g} \exp \pi i
\left(
^t(m+\frac{\e}{2})\tau(m+\frac{\e}{2})+2^t(m+\frac{\e}{2})(z+
\frac{\de}{2})\right).
$$
We   shall use the  notation  $\t_m(\tau, z)$ when the characteristic will be clear from the context.
A {\it characteristic} $m=[\e,\de]$ is called {\it even} or {\it
odd} depending on whether the scalar product $\langle \e,\de \rangle \in\FF_2$
is zero or one. The number of even (resp. odd)
characteristics is $2^{g-1}(2^g+1)$ (resp. $2^{g-1}(2^g-1)$). As a
function of $z$, a theta function is even or odd according to its
characteristic. For $\e\in\FF_2^g$ the {\it second order theta
function with characteristic $\e$} is
$$
\T[\e](\tau,z):=\tt{\e}{0}(2\tau,2z).
$$
These functions are even and they are a basis of $H^0(A_\tau,2\Theta_\tau)$. In general under the map $z\to -z$ the space $H^0(A_\tau,n\Theta_\tau)$, for $n\geq 3$, splits  in two subspaces, $H^0(A_\tau,n\Theta_\tau)^\pm$,  of dimension $\frac{n^g\pm 2^g}{2}$ and $\frac{n^g\pm 1}{2}$ according to the case $n$ even or odd respectively.
 
 Since we shall use theta functions in genus $g$ and $g+1$  we will try to reserve the  notations $\de, \e, \s$
for the genus $g$ case  and $\al, \be$  or  $[*,\e]$ for the genus $g+1$ case.\bigskip

The square of any theta function with characteristic is a section of $2\T_\tau$, and the basis for the space of sections of this bundle is given by theta functions of the second order. Riemann's addition formula is an explicit expression 
 for the product of two theta functions with characteristics (cf. \cite{igusa}, Theorem 2, p. 139). In particular:
\begin{equation*}\label{riem}
 \tt\e\de(\tau,z)^2=\sum\limits_{\sigma\in\FF_2^g}(-1)^{\de\cdot\sigma} \T[\sigma](\tau,0)\T[\sigma+\e](\tau,z).
\end{equation*}
Moreover we  have
\begin{equation*}
\label{tT}
\begin{matrix}
 \T[ \de](\tau,z)\T[ \de+\e](\tau,z)
  =\frac{1}{2^g}\sum\limits_{\s\in\FF_2^g}(-
1)^{ \de\cdot\s}\tt\e{\s}(\tau,2z)\tt\e\s (\tau,0),
\end{matrix}
\end{equation*}
which is valid for all $\tau$, $z$ and $ \de,\e$.
Similarly we have also:
\begin{equation}\label{add}
\tt\e{\delta}(\tau,2z)\tt\e\delta (\tau,0)=\sum\limits_{\sigma\in\FF_2^g}(-1)^{\de\cdot\sigma} \T[\sigma](\tau,z)\T[\sigma+\e](\tau,z).
\end{equation}

We recall a result about  theta functions that  is consequence of Riemann's relations. 
In \cite{gsm}  the  following result has been proved

\begin{lm}\label{t8}
We have the following identity:
$$
 \sum\limits_{\e,\de\in\FF_2^{g}}\theta^8\tch\e\de(\tau,z)
=\sum\limits_{\e,\de\in\FF_2^{g}}\theta^6\tch\e\de(\tau,0)\theta^2\tch\e\de(\tau,2z).
$$
\end{lm}

Theta constants are restrictions of theta functions to $z=0$; we drop the argument $z=0$ in the notations for theta constants.
All theta constants with odd characteristics vanish identically in $\tau$, while theta constants with even characteristics and all theta constants of the second order do not vanish identically.

\section{Modular forms and Codes}

%
%

Here we recall the notion of multiplier system and modular form, see \cite{freitag} for details. 
Let $\Gamma\subset\Gamma_g$ a congruence subgroup, a map $v:\Gamma\to\CC$ is called a {\it multiplier system\/} of weight $r/2$, $r\in\ZZ$ if $v(M)^l=1$ for all $M\in\Gamma$ and for some $l\in\NN$ and defining: 
$$
j_r(M,\tau):=v(M)\det(c\tau+d)^{r/2}, \qquad M=\begin{pmatrix} a&b\\ c&d\end{pmatrix}\in\Gamma, \qquad \tau\in\HH_n
$$
it satisfies the two conditions $j_r(M_1M_2,\tau)=j_r(M_1,M_2\tau)j_r(M_2,\tau)$, i.e. $j_r$ is a cocycle, and $j_r(-1_{2g},\tau)=1$, if $-1_{2g}\in\Gamma$. If $r$ is even there are no ambiguity in the choice of the square root and the multiplier system is a character of $\Gamma$.
Then a {\it modular form} of weight $r/2$ with respect to the multiplier system $v$ is a holomorphic function $f:\HH_n\to\CC$ with the following properties:
\begin{enumerate}[a.]
\item
$
f(M\cdot\tau)=v(M)\det(c\tau+d)^{r/2}f(\tau) \quad \mbox{for all } M\in\Gamma;
$
\item
for every $M\in \Gamma$ the function:
\begin{equation*}
(f|_{r/2}M)(\tau):=\det(c\tau+d)^{-r/2}f(M\tau)=v(M)f(\tau)
\end{equation*}
is bounded in domains of the kind $y\geq y_0 > 0$ with $\tau=x+iy$ and $y_0$ arbitrary.
\end{enumerate}
We denote
$[\Gamma,r/2,v]$ the vector space of such modular forms.
We shall consider  the graded ring
$$A (\Gamma, v):=\bigoplus_{k=0}^{\infty}[\Gamma,k/2, v^{k}]. $$
We omit the multiplier if it is trivial.
\medskip

 All theta constants with characteristics are modular forms of weight one half and suitable multiplier with respect to   $\Gamma(4,8)\subset \SpZ$, while all theta constants of the second order are modular forms of weight one half and a different multiplier $\chi$ with respect to   $\Gamma(2,4)$. This last case  can be formalized as follows:
 there is a  theta map
$$
\T_2: \CC[F_{\e}: \e\in\FF_2 ^{g}]  \to A(\Gamma_{g}[2,4],\,\chi)
$$
sending $F_{\e}$ to $\T[\e]$.
We refer to \cite{runge1}
for the details on the theta map.

We know that
the group
$\Gamma_g$ is generated by the elements
$J=\begin{pmatrix} 0&-1\\ 1&0\end{pmatrix} $ and
$t(S)=\begin{pmatrix} 1&S\\ 0&1\end{pmatrix}  $ for integral symmetric $S$.\smallskip

Moreover it acts on the second order theta-constants. 
We recall the action of the generators, cf. \cite{runge1} for details.
 We denote by $$\vec{\T}_2=(\dots, \T[\e],\dots)$$
  the vector of second  order theta constants.
 Thus we have
$$
\vec{\T}_2 \vert_{\frac12}
t(S)=
D_S \vec{\T}_2
\text{ and }
\vec{\T}_2 \vert_{\frac12}
J=
 T_g\vec{\T}_2,
$$
where
\[
 D_S={\rm diag} (i^{{}^taSa})_{a\in \FF_2^g}
\text{ and }
T_g=\left(\frac{1+i}{2}\right)^g\left((-1)^{\langle a,b \rangle }\right)_{a,b\in \FF_2^g}.
\]
The group 
\[
 H_g=\langle T_g,D_S \text{ integral symmetric }S \rangle\,\subseteq {\rm GL}(2^g,\CC)
\]
is of finite order and 
$$H_g/{\pm1}\simeq
 \Gamma_g/\Gamma_g[2,4]^*.
$$

Here $\Gamma_g[2, 4]^*$ is the subgroup of $ \Gamma_g[2,4]$ defined  by the  condition ${\rm Tr}(a)\equiv  g\,{\rm mod}\, 4$.  The map $\T_2$ results to be $ \Gamma_g/\Gamma_g[2,4]^*$ equivariant.
When $g$ is odd $ \Gamma_g[2,4]$ is the extension of $\Gamma_g[2,4]^*$ by $-1_{2g}$ that  acts trivially.
\smallskip

We see that 
an $H_g$-invariant polynomial goes to a level
one Siegel modular form  of even weight under the map $\T_2$.
We denote by $R_g$ the $H_g$-invariant subring of the homogeneous  polynomials of even  degree in
$ \CC[F_{\al}: \al \in \FF_2^g]$ and $R_g^{m}$ the
vector space of  $H_g$-invariant homogeneous polynomials
of degree $m$, thus we have a  theta map
$$\t_2:R_g\to A(\Gamma_g)$$
whose image is contained in $A(\Gamma_g)^{(2)}$, 
i.e., in the subring of modular forms of even weight.
The map is  surjective when $g\leq 3$,  cf. \cite{runge1}  and \cite{osm}.\smallskip

In  genera 1, 2, the map  is injective too. When $g=3$,  it factorizes on a relation of degree 16: the Schottky relation
 that  can be easily expressed in terms of first order  theta constants, i.e.
  $$S(\tau)= \left(\frac{1}{8}\sum_{m \, even}\theta_m^{16}-(\frac{1}{8}\sum_{m \, even}\theta_m^{8})^2\right).$$
  Obviously  considering  subgroups of  $H_g$ we get maps to  ring of  modular forms relative to groups 
  $\Gamma $  sitting  between $\Gamma_g[2,4]^*$  and $\Gamma_g$.\smallskip
  
   \noindent For our purposes we need to introduce a  space of  modular forms relative to a  subgroup
  of the  modular group containing  $\Gamma_g[2,4]$. We set
  $$
\Gamma_{g, 0}[2]:=\left\lbrace M=\begin{pmatrix} a&b\\ c&d\end{pmatrix}
\in\Gamma_g\, |\,c\equiv
0 \,{\rm mod}\ 2\right\rbrace.
$$
To this group, according to \cite{ru},  corresponds the subgrop $H_{g,4}$ that is generated by the all $D_S$ and  $$AGL(g,\FF _2):=  \FF_2^g\ltimes GL(g,\FF _2),$$ i.e the
affine group of $GL(g,\FF _2)$. It acts on $\FF_2^{g}$ via
$$(\e, M)(\de)= M\de+\e.$$
This action  can be  naturally applied to second order theta constant, permuting them.\smallskip

 We describe the polynomials invariant  with respect to the action of $H_{g,4}$.
We have that a monomial in the  second order theta  constant is admissible
if it is invariant under the subgroup  generated by the $D_S$. This is equivalent to the  fact that the   characteristic appearing satisfy   some  congruences,  we refer to \cite{ru}  for  details.\smallskip

Thus the invariants are obtained  considering polynomials obtained summing on  the $AGL(g,\FF _2)$ orbits of admissible  monomials. An invariant polynomial    determines   a string of $2^g$ number that  are the
 multiplicity of the  second order theta  constants appearing in  an admissible monomial occurring in the polynomial itself. For our case we need the following  
\begin{lm}
 $${\rm dim}\, \left[\Gamma_{3, 0}[2], 6\right]=6.$$
 A basis  is  given by the 
 modular forms  that are polynomials in the second order theta constants, related to the following admissible monomials  
 
$$ (12,0,0,0,0,0,0,0), (8,4, 0,0,0,0,0,0), (4,4,4, 0,0,0,0,0),$$
$$ (6, 2,2,2,0,0,0,0),(4,0,0,0,2,2,2,2), (5,1,1,1,1,1,1,1).$$
\end{lm}
  
\begin{proof} The dimension of the  space is given in \cite{ru}. The  above  monomials  have  distinct orbits and  they are  obviouly independent, since, in genus 3,  the unique relation  between theta constants of the second order  has degree 16.\end{proof}

 \noindent Now we describe modular forms  $F(Z)$ relative to $\Gamma_g$ that  are  polynomials in the $\T[\e]$. We assume  that  four divides the weight $k$.
  If this is the case we know that  such  modular forms are related to  the weight enumerators $W^{g}_C$
  of   Type II code $C\subset \FF_2^{2k}$. A Type II code means a binary self-dual doubly-even code.   To a  weight enumerator is
 associated a theta series
 $$\theta_{\Lambda(C)}(Z)=\sum_{\lambda\in \Lambda(C)^{g+1}}exp(2\pi i{\rm  tr}(^t\lambda Z\lambda))$$
 for the  lattice
$$\Lambda(C) = \left\lbrace\frac{1}{\sqrt 2} x\in\ZZ^{2k} / x\, {\rm  mod}\,2\in C\right\rbrace.$$
  Let  $F(\dots,\T[\al](Z,0), \dots)$  be   as  above, it is a modular form for the full modular group, then, from \cite{smthetanullw}, we  know that
  $$F(Z)=\sum_{C} a_C\theta_{\Lambda(C)}(Z).$$
As  example we recall that  
$$S(\tau)=\theta_{\Lambda(C_1)}(\tau)-\theta_{\Lambda(C_2)}(\tau).$$
Here $ C_1$ and $C_2$ are the two classes of Type II code in degree 16.

\section{Fourier-Jacobi expansion}

We briefly recall Fourier-Jacobi expansion of Siegel modular forms.  To have it in genus $g$, we need to use modular forms of genus $g+1$.  
Let $f(Z)$ be a modular form of weight  $r$ with $2r\in\NN$ and level $l$ with trivial multiplier. We decompose the variable $Z$ in blocks,
$$Z=\begin{pmatrix}w &{}^tz\\ z& \tau\end{pmatrix}, \,\,{\rm with}\, \tau\in\HH_g, \,z\in\CC^{g},\, w\in\HH_1,$$
then, we  have
$$f(Z)=\sum_{n\in 2\NN}\phi_n(\tau, z) e^{\pi i nw/l}.$$

In  the case of theta constants we have, see  \cite{vg}:
\begin{equation*}\label{f1}
  \tt{0\ \e}{\de_1\ \de}(\left(\begin{matrix}w&{}^tz\\ z& \tau\end{matrix}\right))=
 \tt\e\de(\tau,0)+2e^{\pi i\de_1}q^4\tt\e\de(\tau,z)+O(q^{16})
\end{equation*}
and
\begin{equation*}\label{f2}
\tt{1\ \e}{\de_1\ \de}(\left(\begin{matrix}w& {}^tz\\ z& \tau\end{matrix}\right))=2e^{\pi i\de_1/2}q\tt\e\de(\tau,z/2)+O(q^9),
\end{equation*}
where  we let $q:=\exp(\pi i w/4)$.

Similarly we have
\begin{equation*}\label{f3}
  \T[0\ \e](\left(\begin{matrix}w& {}^tz\\ z& \tau\end{matrix}\right)):= \T[\e](\tau,0)+2q^8 \T[\e](\tau,z)+ O(q^{32})\end{equation*}
and

\begin{equation*}\label{f4}
  \T[1\ \e](\left(\begin{matrix}w& {}^tz\\ z& \tau\end{matrix}\right)):= 2q^2  \T[\e](\tau,z/2)
+ O(q^{18}).
\end{equation*}

Thus, if  $F_{g+1}$ is a modular form  of  genus $g+1$ that is a homogeneous polynomial  of degree $2r$ in the $ \T[\al](Z)$, hence a modular form relative to the group $\Gamma_{g+1}[2,4]$, its $q$-expansion is of the form
\begin{equation}\label{f5} 
F_{g+1}(Z)=F^0(\tau)+F^2(\tau,z)q^2+\dots+F^{2n}(\tau,z)q^{2n}+\dots
\end{equation}
If  $F_{g+1}$ is a modular form relative to larger  group, some terms of the   Fourier-Jacobi expansion vanish.  If the group is  $\Gamma_{g+1}$, we have

 $$F_{g+1}(Z)=F^0(\tau)+F^8(\tau,z)q^8+\dots+F^{8n}(\tau,z)q^{8n}+\dots.$$
 When the  modular group is $\Gamma_{g+1}[2]$ we get terms of the form $F^{4n}(\tau,z)$.
From  now on we assume that   $F_{g+1}$ is a modular form relative to   $\Gamma_{g+1}$.
We want to compute the first terms of the $q$-expansion of $F_{g+1}$.
It is clear that $F^0(\tau)$ is the image of $F_{g+1}$ under the Siegel $\Phi$ operator, i.e.
$$F^0(\tau)=\Phi(F_{g+1})(\tau)=lim_{q\to 0}F_{g+1}\left(\begin{matrix}w &{}^tz\\ z& \tau\end{matrix}\right).$$
For any $\tau\in\HH_g$, $F^8(\tau,2z)$ is a section in  $H^0(A_{\tau}, 8\T)^+$.
 Hence we have
 \begin{lm} Let  $F_{g+1}(Z)$,  be a  modular form of  weight $k$ and  $\tau$ be a period  matrix  such that  $\tt{\e}{ \de}(\tau)\neq0$   for all even characteristics, then  $F^{8n}(\tau,2z)$ is  a polynomial  of degree $4n$ in the $ \T[\e](\tau,z)$. The coefficients are meromorphic  modular forms of  weight $k-2n$.
 
\end{lm}
\begin{proof}
It is  an immediate  consequence of the fact that in this case the  Kummer variety $K_{\tau}$ is  normally generated  by $2\T_{\tau}$.  Hence the  map
$$Symm^{4n} (H^0(X_{\tau}, 2\T_{\tau}))\to H^0(X_{\tau}, 8n\T_{\tau})^{+}$$
 is  surjective and  we get  for fixed $\tau$  $F^{8n}(\tau,2z)$ as    a polynomial  of degree $4n$ in the $ \T[\e](\tau,z)$.
Considering $\tau$ as  a variable, the  modularity of $F_{g+1}(Z)$ implies the  modularity of the coefficients.
\end{proof}  

To get  an explicit expression  for   $F^8(\tau,2z)$  we need a further reordering of the monomials appearing in expressing the  modular form $F_{g+1}(Z)$ as a sum of admissible monomials. We  have
$$F_{g+1}(Z)=f_0 (Z)+f_4(Z)+f_8(Z)+\dots$$
Here $f_i$ is the sum of all monomials in the theta constants of the second order  with the number $1$ occurring exactly $i$ times as first entry in the characteristics.

We observe that there is a bijective  map sending   a  polynomial in the $2^{g+1}$ variables $\T[\al](Z)$,
$$p(\T[0,\de](Z), \T[1,\e](Z))$$
to the polynomial in the  $2^{g+1}$ variables  $\T[\de](\tau),\T[\e](\tau, z)$,
$$p(\T[\de](\tau), \T[\e](\tau, z)).$$
 With this notation we have the following

\begin{prop} Let   $F_{g+1}$ be a modular form of weight $k$  that is  a homogeneous polynomial in the $\T[\al](Z)$, then    $F^8(\tau,2z)$ is of the form
$$F^8(\tau,2z) =2^4 f_{4} \left(\dots, \T[\de](\tau),\dots, \T[ \e](\tau, z)\right)+H_1(\tau, z),$$
with
$$H_1(\tau, z)=\sum\limits_{\e\in\FF_2^{g}} \frac{\partial F^{0}}{\partial \T[ \e]} \T[ \e](\tau, 2z).$$
\end{prop}
\begin{proof}
We want to find the coefficient of $q^8$ in the Fourier-Jacobi expansion and this can be obtained as an immediate consequence of the Fourier-Jacobi expansion of the second order theta-constants (cf. \S 5).
Thus, it is clear that we have to consider terms with quartic polynomials in the $\T[ \e](\tau, z/2)$
 and terms  that are  linear in the $\T[ \e](\tau, z)$.These appear only in the Fourier-Jacobi expansion of $f_0$. An easy computation gives $H_1(\tau,z)$.\smallskip
\end{proof}
 We are interested in modular forms of   this type vanishing along $\HH_4$. Really, we restrict our attention  to  modular forms of weight 16. This case has  been    studied in detail in \cite {opy}. It is known that there are  85 weight enumerators  $W^{g+1}_C$ 
   of degree 32. In genus 4 case  they span a 19  dimensional  space. An explicit basis is  given in \cite{opy}.  They are the code polynomials. We denote by $\mathcal C$ the set of related  codes. So  from   \cite{opy}, we learned that the associated theta series span a space of  dimension 14, hence we  have  five  independent relations among them, denoted $\tilde{R}_i$,   that can be expressed in terms of the basis of 19 weight enumerators chosen in \cite{opy} (see also \cite{oweb}) as:
\begin{align*}
\tilde{R}_1 &= -3\mathcal{C}_2+153\mathcal{C}_6-420\mathcal{C}_{10}+640\mathcal{C}_{18}-83\mathcal{C}_{24}-315\mathcal{C}_{25}+28\mathcal{C}_{27} \\
\tilde{R}_2 &= -5\mathcal{C}_{1}+110\mathcal{C}_{2}-616\mathcal{C}_{3}+880\mathcal{C}_{4}+121\mathcal{C}_{5}-121\mathcal{C}_{6}-385\mathcal{C}_{7}+16\mathcal{C}_{23} \\
\tilde{R}_3 &=-63\mathcal{C}_{1}+990\mathcal{C}_{2}-2016\mathcal{C}_{3}-960\mathcal{C}_{4}+957\mathcal{C}_{5}-2610\mathcal{C}_{6}+2520\mathcal{C}_{7}+2520\mathcal{C}_{10} \\
&\phantom{=}+1280\mathcal{C}_{18}-576\mathcal{C}_{23}+498\mathcal{C}_{24}-1890\mathcal{C}_{25}+280\mathcal{C}_{27}-1890\mathcal{C}_{29}+960\mathcal{C}_{67} \\
\tilde{R}_4 &= -944\mathcal{C}_{1}+11597\mathcal{C}_{2}-22624\mathcal{C}_{3}-6080\mathcal{C}_{4}+1252\mathcal{C}_{5}+3269\mathcal{C}_{6}+23660\mathcal{C}_{7}\\
&\phantom{=}-10500\mathcal{C}_{10}-4480\mathcal{C}_{18}-5696\mathcal{C}_{23}-1931\mathcal{C}_{24}+7245\mathcal{C}_{25}-1092\mathcal{C}_{27}\\
&\phantom{=}+22260\mathcal{C}_{29}-37440\mathcal{C}_{67}+21504\mathcal{C}_{82} \\
\tilde{R}_5 &= -64\mathcal{C}_{1}-2041\mathcal{C}_{2}+15400\mathcal{C}_{3}-880\mathcal{C}_{4}+14559\mathcal{C}_{5}+42186\mathcal{C}_{6}-38465\mathcal{C}_{7}\\
&\phantom{=}+50540\mathcal{C}_{10}-117600\mathcal{C}_{11}-225792\mathcal{C}_{16}+3200\mathcal{C}_{18}-26128\mathcal{C}_{23}-63675\mathcal{C}_{24}\\
&\phantom{=}-21315\mathcal{C}_{25}+16716\mathcal{C}_{27}-47985\mathcal{C}_{29}+409600\mathcal{C}_{44}-29760\mathcal{C}_{67}+21504\mathcal{C}_{82},
\end{align*}
where the $\mathcal{C}_i$ are the elements of the basis of the weight enumerators.

   Let $\tilde{R}(Z)$  be one of these relations 
$$\tilde{R}(Z)=\sum_ {C\in\mathcal C_{32}}a_C\theta_{\Lambda(C)}(Z)=0, $$

we consider in  some detail  its   Fourier-Jacobi expansion. Referring to the notation of \eqref{f5}, we have

$$\tilde{R}(Z)=r_0(\tau)+ r_{1}(\tau, z) q^8+ ....$$

Now  $r_0(\tau)$ is  a modular form of weight 16 relative to $\Gamma_3$ vanishing identically on $\HH_3$.
 So it is of the form
$$r_0(\tau)= S(\tau)\left(a\frac{1}{8}\sum_{m \, even}\theta_m^{16}-b(\frac{1}{8}\sum_{m \, even}\theta_m^{8})^2\right),$$
with $a, b \in\CC$.

\begin{lm} We can rewrite the relations $R_1,\dots, R_5$ so that the  corresponding  $r_0(\tau)$
is  equal to $0$ in the first  three cases, $S^2(\tau)$ in the fourth case  and $$S(\tau) (\frac{1}{8}\sum_{m \, even}\theta_m^{8})^2$$ in the last case.

\end{lm}
\begin{proof}
We checked with a computer  that all cases are possible and we got the    relations in terms of the  ``old" relations. The  following is a   solution satisfying the above condition
\begin{align*}
R_1&=\tilde{R}_3-\frac{29481}{500158}\tilde{R}_4-\frac{2582}{2250711}\tilde{R}_5 \\
R_2&=\tilde{R}_2+\frac{32395}{1000316}\tilde{R}_4-\frac{1474}{2250711}\tilde{R}_5 \\
R_3&=\tilde{R}_1-\frac{2097}{250079}\tilde{R}_4-\frac{1996}{2250711}\tilde{R}_5 \\
R_4&=\frac{7987}{8002528}\tilde{R}_4+\frac{77}{6001896\tilde{R}_5} \\
R_5&=-\frac{1}{32}\tilde{R}_1=-W_{e_8}R_0.
\end{align*}
Here $R_0$ denotes a relation involving  the weight enumerators of degree 24. This can be easily deduced  from  table 3 in \cite{opy}, in fact  the weight  enumerators appearing in $\tilde{R}_1$ are  products of the  weight  enumerator $W_{e_8}$ of degree 8 times  weight  enumerators of degree 24. The relation $R_0$ is equivalent to that one that  has been studied in 
 \cite{fro}.
\end{proof}
\begin{prop} In all  these cases    $r_{1}(\tau, z)$ is  a 
quartic polynomial in the  $\T[ \e](\tau, z/2)$    with holomorphic   modular  forms  of  weight  14 as coefficients.
\end{prop}
\begin{proof}
We know that $r_{1}(\tau, z)$  has two summands. The first obviously gives the quartic polynomial and the second (that  occurs only in the case of the fifth relation) we have to consider:
 $$\sum\limits_{\e\in\FF_2^{3}} \frac{\partial r_{0}(\tau)}{\partial \T[ \e]}\T[ \e](\tau, z)= (\frac{1}{8}\sum_{m \, even}\theta_m^{8})^2   \frac{\partial S(\tau)}{\partial \T[ \e]} \T[ \e](\tau, z)              $$  
and by the chain rule:
$$\sum\limits_{\e\in\FF_2^{3}} \frac{\partial r_{0}(\tau)}{\partial \T[ \e]}\T[ \e](\tau, z)=
(\frac{1}{8}\sum_{m \, even}\theta_m^{8})^2\sum\limits_{m\, even}   \frac{\partial S(\tau)}{\partial \theta_m^2}  \theta_m (\tau, z)^2.$$
Hence we need to compute
\begin{align*}
&\sum\limits_{m\, even}   \frac{\partial S(\tau)}{\partial \theta_m^2}  \theta_m (\tau, z)^2=
\sum\limits_{m\, even} \left(\theta_m(\tau)^{14}-(\frac{1}{8}\sum\limits_{n\, even}\theta_n(\tau)^8)\theta_m(\tau)^6\right)\theta_m (\tau, z)^2 \\
&=\sum\limits_{m\, even} \theta_m(\tau)^{14}  \theta_m (\tau, z)^2        -\frac{1}{8}\sum\limits_{n\, even}\theta_n(\tau)^8 \sum\limits_{n\, even}\theta_n^8(\tau,z/2) \\
&=G_1(\tau, z)+G_2(\tau, z).
\end{align*}
Now the proposition follows from \eqref{add} and  lemma \ref{t8}.
\end{proof}
For sake of clarity we write the Fourier-Jacobi expansion of the relations
\begin{align*}
R_1&=q^8 r_{1,1}(\tau, z)+\dots\\
R_2&=q^8 r_{2,1}(\tau, z)+\dots\\ 
R_3&=q^8 r_{3,1}(\tau, z)+\dots\\  
R_4&=  S^2(\tau)+q^8 r_{4,1}(\tau, z)+\dots\\ 
R_5&= S(\tau) (\frac{1}{8}\sum_{m \, even}\theta_m^{8})^2+q^8 r_{5,1}(\tau, z)+\dots\\ 
\end{align*}
Since in  the  previous   discussion  we started  with   $\Gamma_4$-invariant relations we have that  all
 $r_{i, 1}(\tau, 2z)$  has the same structure as  Coble's  quartic, i.e. it is invariant under translation   with points of order two of the  abelian  variety, hence  each of them is of the  form  

$$
r_{i,1}(\tau, 2z)=  s_{i,1}Q_1+\ldots+s_{i,15}Q_{15}, \qquad i=1,\ldots,5,
$$
with

$$Q_1:=x_{000}^4 +x_{001}^4 +x_{010}^4 +x_{100}^4 +x_{110}^4 +x_{101}^4 +x_{011}^4 +x_{111}^4,$$
$$Q_2:=x_{000}^2 x_{001}^2+x_{010}^2 x_{011}^2+x_{100}^2 x_{101}^2+x_{110}^2 x_{111}^2,$$
$$Q_3:=x_{000}^2 x_{010}^2+x_{001}^2 x_{011}^2+x_{100}^2 x_{110}^2+x_{101}^2 x_{111}^2,$$
$$Q_4:=x_{000}^2 x_{011}^2+x_{010}^2 x_{001}^2+x_{100}^2 x_{111}^2+x_{110}^2 x_{101}^2,$$
$$Q_5:=x_{000}^2 x_{100}^2+x_{010}^2 x_{110}^2+x_{001}^2 x_{101}^2+x_{011}^2 x_{111}^2,$$
$$Q_6:=x_{000}^2 x_{101}^2+x_{010}^2 x_{111}^2+x_{100 }^2 x_{001}^2+x_{110}^2 x_{011}^2,$$
$$Q_7:=x_{000}^2 x_{110}^2+x_{010}^2 x_{100}^2+x_{101}^2 x_{011}^2+x_{001}^2 x_{111}^2,$$
$$Q_8:=x_{000}^2 x_{111}^2+x_{010}^2 x_{101}^2+x_{100}^2 x_{011}^2+x_{110}^2 x_{001}^2,$$
$$Q_9:=x_{000} x_{010}x_{100} x_{110}+x_{001} x_{011}x_{101} x_{111},$$
$$Q_{10}:=x_{000} x_{001}x_{100}x_{101}+x_{010} x_{011}x_{110} x_{111},$$
$$Q_{11}:=x_{000} x_{011}x_{100 }x_{111}+x_{001} x_{010}x_{101} x_{110},$$
$$Q_{12}:=x_{000} x_{001}x_{010} x_{011}+x_{100} x_{101} x_{110} x_{111},$$
$$Q_{13}:=x_{000} x_{010}x_{101 }x_{111}+x_{001} x_{011}x_{100} x_{110},$$
$$Q_{14}:=x_{000} x_{001}x_{110}x_{111}+x_{010} x_{011}x_{100} x_{101},$$
$$Q_{15}:=x_{000} x_{011}x_{101 }x_{110}+x_{001} x_{010}x_{100} x_{111}.$$

Where $x_{\e}$ stands for $\T[\e](\tau, z)$. As a consequence of   \cite{gscoble} and  the results in \cite{gsm},   we have that  such quartics are completely determined by
 $$s_1\in[\Gamma_{3,0}[2], 14].$$
 In fact  the coefficients $ s_1, \dots, s_{15}$  span a 15 dimensional  representation of  $\Gamma_3$ that has only a one dimensional $\Gamma_{3,0}[2]$- invariant  space, spanned  by $s_1$.
 
  Now  for each relation $R_i$ in  which occur about 50 millions   admissible monomials, cf. \cite{opy}, we   have  an  expression of   the  corresponding  quartic $r_{i,1}(\tau, z)$, in  which occur about one  million terms.
We  know that these quartics are  determined  by the corresponding $s_{i,1}$ in  which occur about 
5 thousands  admissible monomials,  cf. \cite{sturmfels}.
Really, in \cite{sturmfels}  and \cite{gsm}  the related coefficient $s_1$  of the Coble's  quartic    appears as a monomial of degree 28 in the first order theta-constant.
We  need  a description of $s_1$ as a polynomial in $\T[\e]$. This  is  an immediate consequence of the discussion in \cite{sturmfels}. 

In fact  we have 
$$s_1=\prod_{\e\neq0}\left(\prod_{\de/ \langle \de, \e \rangle =0}\theta\tch0\de(\tau)-\prod_{\al/ \langle \al, \e \rangle =1}\theta\tch0\al(\tau)\right)$$

 and as remarked in \cite{sturmfels} this  is a polynomial in the $\theta\tch\e\de^2(\tau)$  and in
$$\prod_{\de}\theta\tch0\de(\tau).$$

Using Riemann's relations  in genus  three, we get  that also this  term is  a polynomial in the $\theta\tch\e\de^2(\tau)$,
hence we get an  expression of the form
 $$s_1= p_{28}(\dots, \T[\varepsilon],\dots)_{\varepsilon\in\FF_2^3}.$$
It  will have  5360 terms.

Now to each relation  $R_i$, $i=1,\ldots, 5$,  we can associate a
quartic
$$ s_{i,1}Q_1+\ldots+s_{i,15}Q_{15},$$
and we have the following 
\begin{prop}
The modular form $s_1(\tau)$ is contained in the ideal $\mathcal I\subset\CC[\dots ,\T[\e],\dots]_{\e\in\FF_2^3}$ generated  by
$s_{1,1}(\tau), s_{2,1}(\tau),\dots, s_{5,1}(\tau)$.
\end{prop} 
\begin{proof}
A proof  is obtained  using a computer.  The expressions for the coefficients are   very  long, but the  computation involved  does not require much time, once a couple of facts are taken into account. First  we  work on the  polynomial  
ring $\CC[\ldots, X_{\e}, \ldots]_{\e\in\FF_2^3}$. Second, let $F_1,\dots, F_6$ be  a basis of $ [\Gamma_{3,0}[2], 6]$, then instead of checking if  $s_1\in \mathcal I$, we  checked that the associated  polynomial $p_{28}(\ldots, X_{\varepsilon},\ldots)_{\e\in\FF_2^3}$  belongs to the vector  space  $W$  spanned by  the polynomials  associated to 
$$SF_1,SF_2,\dots, SF_6,  s_{1,1}, \dots,  s_{5,1}.$$
We  found also that  a basis is given by the  first 8 vectors and $s_{5,1}$.
Moreover, modulo  the ideal generated  by $S$, i.e. in  $\CC[\ldots, \T[\e], \ldots]_{\e\in\FF_2^3}$ we  have
 $$s_1=-\left(\frac{3787811}{13821931447380}s_{1,1}+ \frac{914993}{4344035597748}s_{2,1}\right).$$
\end{proof}
These results  can be transferred to  quartics vanishing along the universal  Kummer variety of genus 3.  To  have a complete statement, let us  recall
 recall that in 
 \cite{opy}, there is a distinguished relation
\begin{align*}
573102233555{\rm Big Norm}&=    151595494160\tilde{R}_1-292362643392\tilde{R}_2\\&+82765857152\tilde{R}_3+5300722416\tilde{R}_4\\
&+230972544\tilde{R}_5.
\end{align*}
Hence we  have
\begin{prop} From Fourier-Jacobi expansion of  a relation  in genus  4 among theta series related to even self dual  codes of rank 32, one  obtains 3 independent  quartics vanishing along the universal  Kummer variety of genus 3.  A basis is given  by
   Coble's  quartic,   the term in the  expansion of  Big Norm and   the term in the  expansion of $R_5$.
\end{prop}
From the previous proposition we obtain
\begin{cor}
The genus 4 relation, whose $r_1(\tau,z)$ coefficient in the Fourier-Jacobi expansion is the Coble quartic, expressed in terms of the basis chosen in \cite{opy} for the 19 dimensional space of weight enumerators of degree 32 in genus 4 is:
\begin{align*}
\alpha R&=5765253288
   \mathcal{C}_{1}-113833368957
   \mathcal{C}_{2}+290742188352
   \mathcal{C}_{3} \\
   &+12522322560
   \mathcal{C}_{4}-163886691540
   \mathcal{C}_{5}+480649493775
   \mathcal{C}_{6} \\
   &-246978898320
   \mathcal{C}_{7}-465679797660
   \mathcal{C}_{10}-32350348800
   \mathcal{C}_{11} \\
   &-62112669696
   \mathcal{C}_{16}-237874412160
   \mathcal{C}_{18}+54434900352
   \mathcal{C}_{23} \\
   &-111473675885
   \mathcal{C}_{24}+350248142475
   \mathcal{C}_{25}-48264847708
   \mathcal{C}_{27} \\
   &+428125619460
   \mathcal{C}_{29}+112676044800
   \mathcal{C}_{44}-380665602240
   \mathcal{C}_{67} \\
   &+127956347904
   \mathcal{C}_{82},
   \end{align*}
   with $\alpha=608164983684720$.
\end{cor}


 In \cite{sturmfels} it  is  conjectured that the prime ideal   of the universal Kummer variety in genus 3 is generated by 891 bihomogeneous polynomials in $(u, x)$: the Schottky polynomial of degree $(16, 0)$, the eight Coble derivatives of degree $(28, 3)$, and the 882 polynomials of degree $(16, 4)$  that comes  from the  Fourier-Jacobi expansion  of relations among the second order theta constants induced by
Riemann's relations of genus 4. \smallskip

\noindent    Big Norm is the symmetrization of  these relations, so its  Fourier-Jacobi expansion    fits in the conjectural  description of the ideal of the universal  Kummer variety.  From our result, the same is true for  Coble quartic. We still   have  the   Fourier-Jacobi expansion of one  more relation, e.g. $R_5$. In this case  we have that the resulting polynomial of bidegree $(28, 4)$  is reducible, since $R_5$  does it.
 The non vanishing factor,  proportional to $\sum_{m \, even}\theta_m^{8}(\tau)$ has degree $(8,0)$, so we get  a relation of  degree $(20, 4)$ related to the Fourier-Jacobi expansion of the  relation of degree 24, i.e. $R_0$. This  is not obviously  contained in the ideal described in  \cite{sturmfels}. So we  conclude raising the  following question  
 \begin{qu} Is the quartic
$$ \frac{r_{5, 1}(\tau, 2z)}{\sum_{m \, even}\theta_m^{8}(\tau)}$$
 contained in the ideal described  in \cite{sturmfels}?
 \end{qu}

\end{document}